\documentclass[12pt]{amsart}
\usepackage{amssymb, amstext, amscd, amsmath}
\usepackage{mathtools, xypic, color}
\usepackage{enumitem}
\usepackage{xcolor}
\usepackage{hyperref}
\theoremstyle{plain}
\newtheorem{thm}{Theorem}[section]

\newtheorem{lem}[thm]{Lemma}
\newtheorem{cor}[thm]{Corollary}

\theoremstyle{definition}

\mathtoolsset{centercolon}
\newcommand{\bB}{{\mathbb{B}}}
\newcommand{\bC}{{\mathbb{C}}}

\newcommand{\bS}{{\mathbb{S}}}
\newcommand{\bT}{{\mathbb{T}}}

\newcommand{\rC}{\mathrm{C}}

\renewcommand{\phi}{\varphi}

\newcommand{\upchi}{{\raise.35ex\hbox{$\chi$}}}

\newcommand{\ol}{\overline}

\newcommand{\AB}{{\mathrm{A}(\mathbb{B}_d)}}

\newcommand{\Hinf}{{H^\infty(\mathbb{D})}}
\newcommand{\HB}{{H^\infty(\mathbb{B}_d)}}

\newcommand{\FORAL}{\text{ for all }}

\newcommand{\qforal}{\quad\text{for all}\quad}

\newcommand{\re}{\operatorname{Re}}

\newcommand{\TS}{\operatorname{TS}}
\newcommand{\AC}{\operatorname{AC}}

\newcommand{\ext}{\operatorname{ext}}
\newcommand{\wexp}{\operatorname{{\mbox{$w^*$}\!-exp}}}
\begin{document}
\title[Extreme points in the predual of $H^\infty(\bB_d)$]{The unit ball of the predual of $H^\infty(\bB_d)$\\has no extreme points}

\author[R. Clou\^atre]{Rapha\"el Clou\^atre}
\address{Department of Mathematics, University of Manitoba, Winnipeg, MB, Canada}
\email{raphael.clouatre@umanitoba.ca\vspace{-2ex}}
\thanks{The first author is partially supported by an FQRNT postdoctoral fellowship.}

\author[K.R. Davidson]{Kenneth R. Davidson}
\address{Pure Mathematics Dept., University of Waterloo,
Waterloo, ON, Canada}
\email{krdavids@uwaterloo.ca}
\thanks{The second author is partially supported by an NSERC grant.}

\begin{abstract}
We identify the exposed points of the unit ball of the dual space of the ball algebra. As a corollary, we show that the predual of $\HB$ has no extreme points in its unit ball.
\end{abstract}

\maketitle

\section{Introduction} \label{S:intro}

The algebra $\HB$ of bounded analytic functions on the unit ball $\bB_d$ in $\bC^d$ is a weak-$*$ closed subspace 
of $L^\infty(\bS_d,\sigma)$, where $\sigma$ is Haar measure on the unit sphere $\bS_d$ (i.e., the unique rotation invariant probability measure on $\bS_d$).
In particular, it has a predual that can be identified as 
\[
\HB_* \simeq L^1(\bS_d,\sigma)/\HB_\perp
\]
and a corresponding weak-$*$ topology.
In the one dimensional case of the unit disc, Ando \cite{Ando} proved that $\Hinf$ has a unique predual and that the unit ball of this predual has no extreme points.
In this note, we generalize this second fact to higher dimensions: we establish that the unit ball of $\HB_*$ has no extreme points. It is unknown whether $\HB$ has a unique predual when $d>1$, although some of its natural analogues have that property \cite{DW, KY}.

Our proof is not straightforward and is based on a description of the dual space of the ball algebra $\AB$, which consists of analytic functions on $\bB_d$ that are continuous on $\ol{\bB_d}$. 
The idea arose in connection with a larger work \cite{CD} investigating the analogue of the inclusion $\AB\subset \HB$ in the Drury-Arveson space. 
The situation is quite different there as the appropriate version of $\HB$ has many extreme points in the unit ball of its predual. 
Nevertheless, it is the basic arguments used in that context that led to the results of this paper.

The dual of the ball algebra is well understood (see \cite[Chapter 9]{Rudin} for the complete story). 
We briefly review its main features here. 
By the maximum principle, $\AB$ may be viewed as a closed subalgebra of $\rC(\bS_d)$. 
It thus follows from the Hahn-Banach theorem that every functional on $\AB$ has a representing measure on the sphere. 
Accordingly, much information about the dual space is gained from a careful analysis of the representing measures that can arise.
A regular Borel measure $\mu$ on $\bS_d$ is called \emph{Henkin} if every bounded sequence $\{f_n\}_n\subset \AB$ which 
converges to $0$ pointwise on the open ball $\bB_d$ satisfies 
\[ \lim_{n\to\infty} \int f_n \,d\mu = 0 .\]
In that case, a theorem of Valskii \cite{Valskii} shows that there is a measure $\nu \in \AB^\perp$ such that $\mu-\nu$ is absolutely continuous with respect to $\sigma$. 
Hence, the functionals on $\AB$ given by Henkin measures are precisely those  which extend to a weak-$*$ continuous linear functional on $\HB$: integration against $\mu-\nu$ is the desired extension.
In addition, Henkin measures are completely characterized as those which are absolutely continuous with respect to some positive representing measure for the functional of evaluation at the origin \cite{Henkin,ColeRange}. 
Because of this fact, we denote the space of Henkin measures by $\AC$.
At the other extreme, a regular Borel measure on $\bS_d$ is called \emph{totally singular} if it is singular with respect to every positive representing measure for evaluation at the origin. Let $\TS$ be the space of totally singular measures.
The abstract F\&M Riesz Theorem, also known as the Glicksberg-K\"onig-Seever decomposition \cite{Glicksberg, KonigSeever}, 
shows that every measure on $\bS_d$ decomposes uniquely as a sum of a Henkin measure and a totally singular one.

Bringing together all of these results, we obtain the following description of the dual space
\[ \AB^* \simeq \AC/\AB^\perp \oplus_1 \TS \simeq \HB_* \oplus_1 \TS . \]
Since this is a dual space, the set of extreme points of its unit ball is large enough so that the weak-$*$ closed convex hull of that set coincides with the whole unit ball by the Krein-Milman theorem.
On the other hand, the point masses $\lambda\delta_\zeta$ for $\zeta\in\bS_d$ and $|\lambda|=1$ are readily seen to be extreme.  
A standard Hahn-Banach separation argument shows that the weak-$*$ closed convex hull of these point masses is the whole unit ball. 
It is therefore plausible that the unit ball of $\HB_*$ may have no extreme points. 

Our proof is based on some results in convexity.
We use a fact mentioned in an old paper of Klee \cite{Klee}: 
if $E$ is a separable Banach space and $C$ is a weak-$*$ compact convex subset of $E^*$,
then the weak-$*$ closure of the weak-$*$ exposed points of $C$ contains all extreme points 
(see Theorem 4.5 of \cite{Klee} and the remark beginning at the bottom of the same page where he mentions that the corresponding statement is valid in our setting).
Klee claims that this is an easy modification of results of Yosida-Fukamiya \cite{YF} and Milman \cite{Milman}.
As this result does not seem to be familiar even to people working on Banach space and 
we have not found it stated in current books, in Section 2 we provide a proof inspired by \cite{Klee} for the reader's convenience.

With this result in hand, we first show in Section 3 that the weak-$*$ exposed points of the unit ball of $\AB^*$ are
precisely the point masses $\lambda\delta_\zeta$ mentioned above.
Since these form a weak-$*$ compact set, they are in fact the extreme points of the unit ball of $\AB^*$. 
The desired consequence about the absence of extreme points in the unit ball of $\HB_*$ is an immediate consequence.

\section{Exposed points in weak-$*$ convex sets}

Let $C$ be a closed convex set in a locally convex topological vector space $E$. 
An extreme point $x_0$ of $C$ is an \textit{exposed point of $C$} if there is a closed supporting hyperplane 
$H\subset E$ for $C$ at $x_0$ such that $H \cap C = \{x_0\}$. 
When $x_0\neq 0$ and $0\in C$, this is equivalent to the existence of $\phi_0\in E^*$ such that 
\[ \re\phi_0(x) < \re\phi_0(x_0)=1 \qforal x \in C\setminus\{x_0\} .\]
We denote by $\exp(C)$ the set of exposed points of $C$, and by $\ext(C)$ the set of extreme points of $C$.
In particular, if $E$ is a dual space equipped with its weak-$*$ topology and $C$ is a weak-$*$ closed convex set,
then the weak-$*$ exposed points are those extreme points of $C$ which are exposed by weak-$*$ continuous functionals.
They will be denoted by $\wexp(C)$.

A boundary point $x_0$ of $C$ is a \textit{smooth point of $C$} if there is a unique supporting hyperplane $H\subset E$ for $C$ at $x_0$. 
When $x_0\neq 0$ and $0\in C$, this is equivalent to the existence of a unique $\phi_0\in E^*$ such that 
\[ \re\phi_0(x) \le \re\phi_0(x_0)=1  \qforal x \in C .\]

If two locally convex topological vector spaces  $E$ and $F$ are in duality,
then the \emph{polar} of a convex set $C\subset E$ is
\[  C^0 = \{ \phi\in F : \re \phi(x) \le 1 \FORAL x \in C \} .\]
(Warning: one often calls the set $\{ \phi\in F : |\phi(x)| \le 1 \FORAL x \in C \}$ the polar as well.)
A straightforward consequence of the Hahn-Banach theorem shows that if in addition, 
$0 \in C$ and $C$ is closed in the $F$-topology, then $C^{00} = C$. 

There is a certain duality between smooth points of a closed convex set $C$ and exposed points of its polar $C^0$ which we now describe. 
Assume that $0\in C$. 
Suppose that $x_0\neq 0$ is a smooth point of $C$ with unique associated functional $\phi_0 \in F$ satisfying 
\[
 \re \phi_0(x)\leq \re \phi_0(x_0)=1 \qforal x\in C . 
\]
The uniqueness shows that $\phi_0$ is an exposed point of $C^0$, as demonstrated by the functional $x_0$.
Conversely, if $x_0$ is an exposed point of $C$, suppose that $\phi_0 \in C^0$ satisfies 
\[ 
  \re \phi_0(x) < \re \phi_0(x_0)=1 \qforal x\in C . 
\]
Then $\phi_0$ is a smooth point of $C^0$ with unique supporting hyperplane determined by $x_0$.

Another fact that we require is Mazur's smoothness theorem \cite{Mazur} stating that the smooth points of a convex set $C$ 
with non-empty interior coincide with the points of G\^ateaux differentiability of the Minkowski functional, 
and in a separable Banach space $E$ that set is dense in the boundary of $C$ (e.g. see page 171 in \cite{Holmes}).
The following theorem is due to Klee \cite{Klee}.

\begin{thm} \label{T:Klee}
Let $E$ be a separable Banach space.
Suppose that $C$ is a weak-$*$ compact convex subset of $E^*$.
Then the weak-$*$ closed convex hull of $\wexp(C)$ is $C$, and the weak-$*$ closure of $\wexp(C)$ contains all extreme points.
\end{thm}

\begin{proof}
Consider the dual pair $(E, E^*)$, where $E^*$ is endowed with the weak-$*$ topology. 
We may assume that $0$ belongs to $C$.
Let $K\subset E^*$ be the weak-$*$ closed convex hull of $\wexp(C)$.
By virtue of the duality mentioned above, to achieve the equality $K=C$, it suffices to prove that $K^0=C^0$. 
Now, $K \subset C$, whence $C^0 \subset K^0$.

Suppose that there is a point $e_0 \in K^0 \setminus C^0$.
Since $C$ is weak-$*$ compact, it must be norm bounded.
Hence there is an $R>0$ such that $\|y\|\leq R$ for every $y\in C$. 
In particular, we see that $e\in C^0$ whenever $\|e\|<1/R$. 
Thus $C^0$ has non-empty interior.
Let $U_0$ be the algebraic convex hull of $\{ e\in E : \|e\|<1/R\}$ and $\{e_0\}$. 
Then, it is easily verified that the set $U=U_0\setminus \{e_0\}\subset K^0$ is open.

Since $0\in C^0$, $e_0\notin C^0$ and $C^0$ is closed and convex, there exists $0< t<1$ such that $te_0$ belongs to the boundary $B$ of $C^0$.  
Thus, $B\cap U$ is a non-empty relatively open subset of $B$. 
Invoking Mazur's smoothness theorem, we see that there is a smooth point $e_1\neq 0$ of $C^0$ in $U \cap B$.
By the remarks preceding the proof, the unique functional $\phi_1 \in C^{00}=C$ which satisfies
\[
 \re \phi_1(x)\leq \re \phi_1(e_1)=1 \qforal x\in C^0 . 
\]
is an exposed point of $C$.
Therefore $\phi_1$ belongs to $K$ by definition, and trivially $\phi_1\in \wexp(K)$.
Using the remark again, we see that $e_1$ is a smooth point of $K^0$.
However this is impossible because $e_1\in U$ lies in the interior of $K^0$.
Thus we must have $K=C$ as claimed.
 
The last statement now follows from the usual converse to the Krein-Milman Theorem (e.g. see \cite[Theorem~V.7.8]{Conway}).
\end{proof}

In our application of this result, $C$ will be the unit ball of $E^*$.
Since in this case $C$ is balanced, we could use the other definition of polar instead.
The details of the proof remain the same.

\section{The predual of $\HB$}

Following the plan outlined in the introduction, we wish to find all exposed points of the unit ball of $\AB^*$. Let us first deal with the extreme points. For a Banach space $X$, we denote by $b_1(X)$ its open unit ball.

\begin{lem} \label{extreme}
The extreme points of $\ol{b_1(\AB^*)}$ decompose as
\[ \ext(\ol{b_1(\AB^*)}) = \ext(\ol{b_1(\HB_*)}) \cup \ext(\ol{b_1(\TS)}) .\]
\end{lem}

\begin{proof}
This essentially follows from the $\ell_1$-decomposition $\AB^* = \HB_* \oplus_1 \TS$. 
Indeed, we see that $\ol{b_1(\AB^*)}$ is the convex hull of $\ol{b_1(\HB_*)} \cup \ol{b_1(\TS)}$, whence
extreme points of $\ol{b_1(\AB^*)}$ must lie in $\ext(\ol{b_1(\HB_*)}) \cup \ext(\ol{b_1(\TS)})$. 
Conversely, let $\phi\in \ext(\ol{b_1(\HB_*)})$ and assume that
\[
\phi=\frac{1}{2}\psi+\frac{1}{2}\theta
\]
for some $\psi,\theta\in \ol{b_1(\AB^*)}$. 
Write $\psi=\psi_1+\psi_2$ and $\theta=\theta_1+\theta_2$, 
where $\psi_1,\theta_1\in \ol{b_1(\HB_*)}$ and $\psi_2,\theta_2\in \ol{b_1(\TS)}$. 
Since $\HB_*\cap \TS=\{0\}$, we must have
\[
\phi=\frac{1}{2}\psi_1+\frac{1}{2}\theta_1
\]
and $\psi_2=\theta_2=0$. By choice of $\phi$, this implies $\phi=\psi=\theta$ which shows that $\phi\in \ext(\ol{b_1(\AB^*)})$. 
A similar argument shows that $\ext(\ol{b_1(\TS)})\subset \ext(\ol{b_1(\AB^*)})$ and the proof is complete.
\end{proof}

We can now give a description of the weak-$*$ exposed points of the unit ball of $\AB^*$. Let $\bT$ be the unit circle.

\begin{thm} \label{T:exposed}
The weak-$*$ exposed points of  $\ol{b_1(\AB^*)}$ are the point masses $\lambda\delta_\zeta$ for $\zeta\in\bS_d$ and $\lambda\in \bT$.
\end{thm}

\begin{proof}
First we show that the point masses are exposed. 
Let $f(z)=(1+z_1)/2$ for every $z\in \bB_d$. 
Then $\|f\|_{\infty}=1$ and $\delta_{\zeta_1}(f)=1$ where $\zeta_1=(1,0,\ldots,0)\in \bS_d$.
Let now $\phi\in\ol{b_1(\AB^*)}$. 
By the Hahn-Banach theorem there is a measure $\mu$ on $\bS_d$ with $\|\mu\|=\|\phi\|$ such that
\[
\phi(g)=\int_{\bS_d}gd\mu \qforal g\in \AB.
\]
In view of the inequalities
\[
 \left|\int_{\bS_d} f\ d\mu\right| \leq \int_{\bS_d}\frac{|1+z_1|}{2}d|\mu| \leq  \|\mu\|\leq 1
\]
we see that $\re \phi(f)=1$ only if $\mu$ is positive, supported on $\{\zeta_1\}$ and $\|\mu\|=1$, which forces $\phi=\delta_{\zeta_1}$. 
We conclude that $\delta_{\zeta_1}$ is exposed. 
Scalar multiples of $\delta_{\zeta_1}$ are dealt with similarly. 
By symmetry, we see that all other point masses $\lambda \delta_{\zeta}$ are exposed points of $\ol{b_1(\AB^*)}$.

Next consider $\mu \in\TS$ with $\|\mu\|=1$ which is not a point mass. 
Then, there exists a Borel set $A\subset \bS_d$ such that $0<|\mu|(A)<1$. Let $\mu_A$ 
and $\mu_{\bS_d\setminus A}$ be the restrictions of $\mu$ to $A$ and $\bS_d\setminus A$ respectively. 
In particular, these measures are both absolutely continuous with respect to $\mu$ and thus lie in $\TS$. Moreover, 
\[
\frac{1}{|\mu|(A)}\mu_A\in \ol{b_1(\TS)} \quad \text{and} \quad \frac{1}{|\mu|(\bS_d\setminus A)}\mu_{\bS_d\setminus A}\in \ol{b_1(\TS)} .
\]
Since $|\mu|(\bS_d)=\|\mu\|=1$, the decomposition 
\[
\mu=|\mu|(A)\left(\frac{1}{|\mu|(A)}\mu_A \right)+|\mu|(\bS_d\setminus A)\left(\frac{1}{|\mu|(\bS_d\setminus A)}\mu_{\bS_d\setminus A} \right)
\]
shows that $\mu$ is not an extreme point of $\ol{b_1(\TS)}$.
By Lemma \ref{extreme}, we see that $\mu$ is not an extreme point of the unit ball of $\AB^*$,
and thus is certainly not a weak-$*$ exposed point.

Finally, consider $\phi\in \HB_*$ with $\|\phi\|=1$.
If $\phi$ were a weak-$*$ exposed point, there would be a function $f \in \AB$ with $\|f\|_\infty =1$ such that
\[ 1 = \re \phi(f) > \re \psi(f) \qforal \psi\in \ol{b_1(\AB^*)}\setminus\{\phi\} .\]
However this is impossible because $f$ attains its norm at some point $\zeta\in\bS_d$. 
Indeed, if we suppose that $f(\zeta) = \lambda \in\bT$, then $(\bar\lambda \delta_\zeta)(f)=1$ while $\bar\lambda \delta_{\zeta}\neq \phi$. 

By virtue of Lemma \ref{extreme}, we see that the only weak-$*$ exposed points are point masses.
\end{proof}

We can now prove the statement in the title of the paper.

\begin{cor} \label{C:none}
The extreme points of  $\ol{b_1(\AB^*)}$ are the point masses $\lambda\delta_\zeta$ for $\zeta\in\bS_d$ and $\lambda\in \bT$.
In particular,  $\ol{b_1(\HB_*)}$ has no extreme points.
\end{cor}

\begin{proof}
By Theorem~\ref{T:exposed}, the set of weak-$*$ exposed points of $\ol{b_1(\AB^*)}$  is 
\[
\{\lambda\delta_\zeta : \zeta\in\bS_d,\ \lambda\in\bT\}.
\]
It is easy to verify that when it is endowed with the weak-$*$ topology of $\AB^*$, 
this set is homeomorphic to $\bS_d\times\bT$ and thus is weak-$*$ compact.
By Theorem~\ref{T:Klee}, this set must therefore coincide with the set of extreme points of $\ol{b_1(\AB^*)}$. 
By Lemma \ref{extreme}, we see that $\ol{b_1(\HB_*)}$ has no extreme points as point masses lie in $\TS$.
\end{proof}

\bibliographystyle{amsplain}

\end{document}